\documentclass[a4paper,12pt,reqno]{amsart}

\usepackage[english]{babel}
\usepackage[T1]{fontenc}
\usepackage[left=1.1in,right=1.1in,top=1.2in,bottom=1.2in]{geometry}
\usepackage{times}
\usepackage{microtype}

\usepackage{commath,amssymb,amscd,mathrsfs,mathtools,dsfont,bbm,multirow,array,caption}
\usepackage[all]{xy}
\usepackage{enumerate}
\usepackage{longtable}
\usepackage{comment}

\usepackage[usenames,dvipsnames,svgnames,table]{xcolor}
\definecolor{red}{RGB}{255,25,25}
\definecolor{blue}{RGB}{25,50,200}
\usepackage{tikz-cd}
\usetikzlibrary{decorations.pathmorphing}

\usepackage[pagebackref,linktocpage,dvipdfmx]{hyperref}

\hypersetup{
pdftitle={\@title},			
pdfauthor={\authors},		
colorlinks=true,				
linkcolor=red,				
citecolor=MidnightBlue,		
filecolor=magenta,			
urlcolor=MidnightBlue			
}

\usepackage{cleveref} 

\newtheorem{theorem}{Theorem}[section]
\crefname{theorem}{Theorem}{Theorems}
\newtheorem{lemma}[theorem]{Lemma}
\crefname{lemma}{Lemma}{Lemmas}

\crefname{proposition}{Proposition}{Propositions}
\newtheorem{prop}[theorem]{Proposition}
\crefname{prop}{Proposition}{Propositions}

\crefname{corollary}{Corollary}{Corollaries}

\crefname{cor}{Corollary}{Corollaries}

\crefname{conjecture}{Conjecture}{Conjectures}
\newtheorem{conj}[theorem]{Conjecture}
\crefname{conj}{Conjecture}{Conjectures}
\newtheorem*{conj*}{Conjecture}
\crefname{conj}{Conjecture}{Conjectures}

\theoremstyle{definition}

\crefname{definition}{Definition}{Definitions}
\newtheorem{defn}[theorem]{Definition}
\crefname{defn}{Definition}{Definitions}

\crefname{example}{Example}{Examples}

\crefname{notation}{Notation}{Notation}
\newtheorem*{notation*}{Notation}
\crefname{notation}{Notation}{Notation}

\crefname{problem}{Problem}{Problems}
\newtheorem{question}[theorem]{Question}
\crefname{question}{Question}{Questions}

\crefname{condition}{Condition}{Conditions}

\crefname{assumption}{Assumption}{Assumptions}

\theoremstyle{remark}
\newtheorem{rmk}[theorem]{Remark}
\crefname{rmk}{Remark}{Remarks}
\newtheorem*{rmk*}{Remark}
\crefname{rmk}{Remark}{Remarks}

\crefname{remark}{Remark}{Remarks}

\crefname{fact}{Fact}{Facts}

\crefname{claim}{Claim}{Claims}
\newtheorem*{claim*}{Claim}
\crefname{claim}{Claim}{Claims}

\crefname{step}{Step}{Steps}

\crefname{case}{Case}{Cases}

\numberwithin{equation}{section}

\newcommand{\bbC}{\mathbb{C}}

\newcommand{\bbR}{\mathbb{R}}

\newcommand{\bR}{\mathbb{R}}

\newcommand{\NE}{\overline{\operatorname{NE}}}

\begin{document}

\title[Bounding cohomology]{Bounding cohomology on a smooth projective surface with Picard number 2}

\author{Sichen Li}
\address{School of Mathematical Sciences, Fudan University, 220 Handan Road, Yangpu District, Shanghai 200433, People's Republic of China}
\email{\href{mailto:sichenli@fudan.edu.cn}{sichenli@fudan.edu.cn}}
\urladdr{\url{https://www.researchgate.net/profile/Sichen_Li4}}
\begin{abstract}
The following conjecture arose out of discussions between B. Harbourne, J. Ro\'e, C. Cilberto and R. Miranda: for a smooth projective surface $X$ there exists a positive constant $c_X$ such that $h^1(\mathcal O_X(C))\le c_X h^0(\mathcal O_X(C))$ for every prime divisor $C$ on $X$.
When the Picard number $\rho(X)=2$, we prove that if either the Kodaira dimension $\kappa(X)=1$ and $X$ has a negative curve or $X$ has two negative curves, then this conjecture holds for $X$.
\end{abstract}




\maketitle


\section{Introduction}
In this note we work over the field $\bbC$ of complex numbers.
By a $(negative)~ curve$ on a surface we will mean a reduced, irreducible curve (with negative self-intersection).
By a $(-k)$-$curve$, we mean a negative curve $C$ with $C^2=-k<0$.

The bounded negativity conjecture (BNC for short)  is one of the most intriguing problems in the theory of projective surfaces and can be formulated as follows.
\begin{conj}\cite[Conjecture 1.1]{B.etc.13}\label{BNC}
For a smooth projective surface $X$ there exists an integer $b(X)\ge0$ such that $C^2\ge-b(X)$ for every curve $C\subseteq X$.
\end{conj}
Let us say that a smooth projective surface  $X$ has
\begin{equation*}
b(X)>0
\end{equation*}
 if there is at least one negative curve on $X$.

The main aim of the short note is to study the following conjecture, which implies BNC (cf. \cite[Proposition 14]{C.etc.17}).
\begin{conj}\cite[Conjecture 2.5.3]{B.etc.12}\label{BH}
Let $X$ be a smooth projective surface.
Then there exists a constant $c_X>0$ such that  $h^1(\mathcal O_X(C))\le c_Xh^0(\mathcal O_X(C))$ for every curve $C$ on $X$.
\end{conj}
On the other hand, the authors of \cite{B.etc.12} disproved  \cref{BH} by giving a counterexample with a large Picard number (cf. \cite[Corollary 3.1.2]{B.etc.12}).
 However, they pointed out that it could still be true that \cref{BH} holds when restricted to rational surfaces (cf. \cite[Proposition 3.1.3]{B.etc.12}).
On the one hand, there exists a new evidence of BNC (cf. \cite[Theorem 1.6]{Li19}).
 This motivates us to consider that whether \cref{BH} is true for $X$ when the Picard number  $\rho(X)=2$ and $b(X)>0$.

Below is our main theorem.
\begin{theorem}\label{Main-thm}
Let $X$ be a smooth projective surface with Picard number 2.
If the Kodaira dimension $\kappa(X)=1$ and $b(X)>0$ or $X$ has two negative curves, then \cref{BH} holds for $X$.
\end{theorem}
\begin{rmk}
In \cite[Claim 2.11]{Li19},  we give a classification of a smooth projective surface  $X$ with $\rho(X)=2$ and two negative curves.
\end{rmk}
\section{Preliminaries}
We first recall  the following question posed in \cite{C.etc.17}.
\begin{question}\cite[Question 4]{C.etc.17}\label{Que4}
Does there exist a constant $m(X)$ such that $\frac{(K_X\cdot D)}{D^2}<m(X)$ for any effective divisor $D$ with $D^2>0$ on a smooth projective surface  $X$?
\end{question}
 If  \cref{BH} is true for a smooth projective surface $X$, then $X$ is affirmative for  \cref{Que4} (cf. \cite[Proposition 15]{C.etc.17}).
This motivates us to give the following definition.
\begin{defn}
Let $X$ be a smooth projective surface.
\begin{enumerate}
 \item[(1)] For every $\bbR$-divisor $D$ with $D^2\ne0$ on $X$, we define a value of $D$ as follows:
\begin{equation*}
                                                       l_D:=\frac{(K_X\cdot D)}{\max\bigg\{ 1, D^2\bigg\}}.
\end{equation*}
\item[(2)] For every $\bbR$-divisor  $D$ with $D^2=0$ on $X$, we define a value of $D$ as follows:
\begin{equation*}
                           l_D:=\frac{(K_X\cdot D)}{\max\bigg\{1,h^0(\mathcal O_X(D))\bigg\}}.
\end{equation*}
\end{enumerate}
\end{defn}
The following is a numerical characterization of \cref{BH}.
\begin{prop}\label{MainProp}
Let $X$ be a smooth projective surface.
If $X$ satisfies the BNC and there exists a positive constant $m(X)$ such that $l_C\le m(X)$ for every curve $C$ on $X$ and  $D^2\le m(X)h^0(\mathcal O_X(D))$ for every curve $D$ with $l_D>1$ and $D^2>0$ on $X$, then $X$ satisfies  \cref{BH}.
\end{prop}
\begin{proof}
Take a curve $C$ on $X$. Note that by Serre duality (cf. \cite[Corollary III.7.7 and III.7.12]{Hartshorne77}), $h^2(\mathcal O_X(C))=h^0(\mathcal O_X(K_X-C))\le p_g(X)$.
As a result,
\begin{equation}\label{eq1}
h^2(\mathcal O_X(C))-\chi(\mathcal O_X)\le  q(X)-1.
\end{equation}
Here, $p_g(X)$ and $q(X)$ are the geometric genus and the irregularity of $X$ respectively.

Our main condition is the following:

(*) \label{*} There exists a positive constant $m(X)$ such that $l_C\le m(X)$ for every curve $C$ on $X$ and  $D^2\le m(X)h^0(\mathcal O_X(D))$ for every curve $D$ with $l_D>1$ and $D^2>0$ on $X$.

We divide the proof into the following three cases.

Case~(i).
Suppose $C^2>0$. Then by Riemann-Roch theorem (cf. \cite[Theorem V.1.6]{Hartshorne77}),
\begin{equation}\label{eq2}
                                             h^1(\mathcal O_X(C))=h^0(\mathcal O_X(C))+h^2(\mathcal O_X(C))-\chi(\mathcal O_X)+\frac{C^2(l_C-1)}{2}.
\end{equation}
If $l_C\le1$, then Equation (\ref{eq1}) and (\ref{eq2}) imply  that $h^1(\mathcal O_X(C))\le h^0(\mathcal O_X(C))+q(X)-1$, which is the desired result by $c_X:=q(X)$.
If $l_C>1$, then Equation (\ref{eq1}) and (\ref{eq2})  and the condition (*)  imply that $2h^1(\mathcal O_X(C))\le (m^2(X)-m(X)+2)h^0(\mathcal O_X(C))+2(q(X)-1)$,
which is the desired result by $2c_X:=m^2(X)-m(X)+2q(X)$.

Case~(ii).
Suppose $C^2=0$. Then by Riemann-Roch theorem,
\begin{equation}\label{eq3}
                                             2h^1(\mathcal O_X(C))=2h^2(\mathcal O_X(C))-2\chi(\mathcal O_X)+h^0(\mathcal O_X(C))(l_C+2),
\end{equation}
which, Equation (\ref{eq1}) and the condition (*) imply that
\begin{equation*}
                                             2h^1(\mathcal O_X(C))\le 2(q(X)-1)+h^0(\mathcal O_X(C))(m(X)+2),
\end{equation*}
which is the desired result by $2c_X:=m(X)+2q(X)$.

Case~(iii).
Suppose $C^2<0$. Then $h^0(\mathcal O_X(C))=1$. Since $X$ satisfies the BNC, there exists a positive constant $b(X)$ such that every curve $C$ on $X$ has $C^2\ge-b(X)$. By Riemann-Roch theorem,
\begin{equation}\label{eq4}
                                             2h^1(\mathcal O_X(C))=2+2h^2(\mathcal O_X(C))-2\chi(\mathcal O_X)+l_C-C^2,
\end{equation}
which, Equation (\ref{eq1}) and  the condition (*)  imply that $2h^1(\mathcal O_X(C))\le 2q(X)+m(X)+b(X)$, which is the desired result by $2c_X:=2q(X)+m(X)+b(X)$.

In all, we  complete the proof of \cref{MainProp}.
\end{proof}
\section{The proof of \cref{Main-thm}}
The proof of \cref{Main-thm} is motivated by the following result.
\begin{prop}\label{>0}
Let $X$ be a smooth projective surface with $\rho(X)=2$. Then the following statements hold.
\begin{enumerate}
\item[(i)] $\NE(X)=\bbR_{\ge0}[f_1]+\bbR_{\ge0}[f_2]$, $f_1^2\le0, f_2^2\le0$ and $f_1\cdot f_2>0$. Here, $f_1, f_2$ are extremal rays.
\item[(ii)] If a curve $C$ has $C^2\le0$, then $C\equiv af_1$ or $C\equiv bf_2$ for some $a,b\in\mathbb R_{>0}$.
\item[(iii)] Suppose a divisor $D\equiv a_1f_1+a_2f_2$ with $a_1,a_2>0$ in (i).  Then $D$ is big. Moreover, if $D$ is a curve, then $D$ is nef and big and $D^2>0$.
\end{enumerate}
\end{prop}
\begin{proof}
By \cite[Lemma 1.22]{KM98}, (i) and (ii) are clear since $\rho(X)=2$. For (iii), $D\equiv a_1f_1+a_2f_2$ with $a_1, a_2>0$ is an interior point of Mori cone, then by \cite[Theorem 2.2.26]{Lazarsfeld04}, $D$ is big. Moreover, if $D$ is a curve, then $D$ is nef. As a result, $D^2>0$.
\end{proof}
\begin{lemma}\label{kappa(X)=1}
Let $X$ be a smooth projective surface with $\rho(X)=2$.
If $\kappa(X)=1$ and $b(X)>0$, then $X$ satisfies  \cref{BH}.
\end{lemma}
\begin{proof}
Since $\kappa(X)=1, \rho(X)=2$ and $\kappa(X)$ is a birational invariant, $K_X$ is nef and semi-ample.
By \cite[Proposition IX.2]{Beauville96}, we have $K_X^2=0$ and there is a surjective morphism $p: X\to B$ over a smooth curve $B$, whose general fibre $F$ is an elliptic curve.
Since $b(X)>0$, $X$ has exactly one negative curve $C$ by \cite[Claim 2.14]{Li19}.
In fact, $p$ is an Iitaka fibration of $X$.
In \cite{Iitaka70}, S. Iitaka proved that if $m$ is any natural number divisible by 12 and $m\ge86$, then $|mK_X|$ defines the Iitaka fibration.
Hence, there exists a curve $F$ as a general fiber of $p$ such that $F\equiv mK_X$.
Then by \cref{>0}(i)(ii), $\NE(X)=\bbR_{\ge0}[F]+\bbR_{\ge0}[C]$.
Note that $(F\cdot C)>0$ since $\rho(X)=2$.
Take a curve $D\equiv a_1F+a_2C$ with $a_1,a_2\ge0$.
 By  \cref{>0}(iii), $D^2>0$ if and only if $a_1,a_2>0$, $D^2=0$ if and only if $D\equiv a_1F$.

 Now suppose $D\equiv a_1F$. Then $l_D=0$.
 Note that $h^1(\mathcal O_X(D))\le q(X)h^0(\mathcal O_X(D))$ by Riemann-Roch theorem and Equation (\ref{eq1}).
 This ends the proof of this case.

 Now suppose $D^2>0$. Then $(F\cdot D)\ge1$ and $(C\cdot D)\ge0$, which imply that
\begin{equation}\label{eq6}
                                                 a_2\ge(F\cdot C)^{-1}, a_1\ge a_2(-C^2)(F\cdot C)^{-1}.
\end{equation}
Therefore, by Equation (\ref{eq6}),
\begin{equation*}\begin{split}
                        l_D&=\frac{(F\cdot D)}{m(a_1(F\cdot D)+a_2(C\cdot D))}\\&\le (F\cdot C)^2(-mC^2)^{-1}.
\end{split}\end{equation*}
Hence, there exists a positive constant $m(X)$ such that $l_D\le m(X)$ for every curve $D$ on $X$.
If $ma_1\ge1$, then  $(K_X-D)D=(1-ma_1)(K_X\cdot D)-a_2(C\cdot D)\le0$.
As a result, $h^1(\mathcal O_X(D))=q(X)h^0(\mathcal O_X(C))$ by Riemann-Roch theorem and Equation (\ref{eq1}).
If $ma_1<1$, then by Equation (\ref{eq6}), $a_2<(F\cdot C)(-mC^2)^{-1}$. So $D^2<2m^{-2}(F\cdot C)^2(-C^2)^{-1}$. Hence, by  \cref{MainProp}, $X$ satisfies  \cref{BH}.
\end{proof}
Now we give a useful result of the Nef cone when $\rho(X)=2$.
\begin{prop}\label{nef}
Let $X$ be a smooth projective surface with $\rho(X)=2$. If $X$ has two negative curves $C_1$ and $C_2$, then the nef cone $\mathrm{Nef}(X)$ is
\begin{equation*}
 \mathrm{Nef}(X)=\bigg\{ a_1C_1+a_2C_2\bigg| a_1(C_1\cdot C_2)\ge a_2(-C_2^2), a_2(C_1\cdot C_2)\ge a_1(-C_1^2), a_1>0, a_2>0\bigg\}.
\end{equation*}
\end{prop}
\begin{proof}
Since $\rho(X)=2$, $\NE(X)=\bbR_{\ge0}[C_1]+\bbR_{\ge0}[C_2]$  by \cref{>0}(ii).
As a result, an effective $\bR$- divisor $D\equiv a_1C_1+a_2C_2$ is nef if and only if $D\cdot C_1\ge0$ and $D\cdot C_2\ge0$, which imply the desired result.
\end{proof}
\begin{lemma}\label{twoneg}
Let $X$ be a smooth projective surface with $\rho(X)=2$.
If $X$ has two negative curves $C_1$ and $C_2$, then $X$ satisfies  \cref{BH}.
\end{lemma}
\begin{proof}
Note that $\NE(X)=\bbR_{\ge0}[C_1]+\bbR_{\ge0}[C_2]$ by \cref{>0}(ii).
We first show that there exists a positive constant $m(X)$ such that $l_C\le m(X)$ for every curve on $X$.
By \cite[Claim 2.11]{Li19}, $\kappa(X)\ge0$, i.e., there exists a positive integral number $m$ such that $h^0(X, \mathcal O_X(mK_X))\ge0$.
Therefore, $K_X$ is $\mathbb Q$-effective  divisor.
As a result, $K_X\equiv  aC_1+bC_2$ with $a, b\in\bbR_{\ge0}$.
Take a curve $D\equiv a_1C_1+a_2C_2$ with $a_1, a_2>0$, then by  \cref{>0}(iii), $D^2>0$.
As a result, $D\cdot C\ge0$ and $X$ has no any curves with zero self-intersection.
$D^2\ge1$ implies that either $D\cdot C_1\ge1$ and $D\cdot C_2\ge0$ or $D\cdot C_1\ge0$ and $D\cdot C_2\ge1$.
Without loss of generality, suppose that $D\cdot C_2\ge0$ and $D\cdot C_1\ge1$.
Then $a_1\ge (C_1^2+(C_1\cdot C_2)^2(-C_2^2)^{-1})^{-1}$. Here,  $C_1^2+(C_1\cdot C_2)^2(-C_2^2)^{-1}>0$ since $\rho(X)=2$. By symmetry and \cref{nef},
\begin{equation*}
                                               a_i\ge c:=\min\bigg\{(C_i^2+\frac{(C_1\cdot C_2)^2}{-C_j^2})^{-1}, \frac{-C_j^2}{(C_1\cdot C_2)}(C_i^2+\frac{(C_1\cdot C_2)^2}{-C_j^2})^{-1}\bigg\},
\end{equation*}
where $i\ne j\in\{1,2\}$. Therefore,
\begin{equation*}\begin{split}
                                                    l_D&=\frac{a(D\cdot C_1)+b(D\cdot C_2)}{a_1(D\cdot C_1)+a_2(D\cdot C_2)}\\&\le\max\bigg\{\frac{a}{c}, \frac{b}{c}\bigg\}.
\end{split}\end{equation*}
So  there exists a positive constant $m(X)$ such that $l_D\le m(X)$ for every curve $D$ on $X$.

 If $a_1>a$ and  $a_2>b$, then
 \begin{equation*}
 (K_X-D)D=(a-a_1)(D\cdot C_1)+(b-a_2)(D\cdot C_2)<0.
 \end{equation*}
This and Equation (\ref{eq1}) imply that
\begin{equation*}\begin{split}
                                             h^1(\mathcal O_X(D))&=h^0(\mathcal O_X(D))+h^2(\mathcal O_X(D))+\frac{(K_X\cdot D)-D^2}{2}-\chi(\mathcal O_X)\\&\le q(X)h^0(\mathcal O_X(D)).
\end{split}\end{equation*}
If $a_1\le a$ or  $a_2\le b$, then by \cref{nef}, $a_2\le a(C_1\cdot C_2)(-C_2^2)^{-1}$ or $a_1\le b(C_1\cdot C_2)(-C_1^2)^{-1}$.
As a result,
\begin{equation*}
D^2\le \max\bigg\{2a^2(C_1\cdot C_2)^2(-C_2^2)^{-1},2b^2(C_1\cdot C_2)^2(-C_1^2)^{-1}\bigg\}.
\end{equation*}
Therefore, $X$ satisfies \cref{BH}  by \cref{MainProp}.
\end{proof}
\begin{proof}[Proof of \cref{Main-thm}]
It follows from \cref{kappa(X)=1,twoneg}.
\end{proof}
We end by asking the following questions.
\begin{question}
Is \cref{BH} true for any smooth projective surface $X$ with $\rho(X)=2$ and $b(X)>0$?
\end{question}
\begin{question}
Let $X$ be a smooth projective surface with $\rho(X)=r$ and $C_i$ some curves on $X$.
Suppose that $\NE(X)=\sum_{i=1}^r\bR_{\ge0}[C_i]$.
Is \cref{BH} true for $X$?
\end{question}

\section*{Acknowledgments}
The author would like to thank Prof. Meng Chen, Prof. Rong Du and Prof. De-Qi Zhang  for their constant encouragement and the anonymous referee for several suggestions.
The author is supported  by the National Natural Science Foundation of China (Grant No. 12071078).

\end{document}